\documentclass[12pt]{article}

\usepackage{latexsym}
\usepackage{amsmath}
\usepackage{amsthm}
\usepackage{amsfonts}
\usepackage{eufrak}
\usepackage{amssymb}
\usepackage{verbatim}
\usepackage{graphicx, color,mathrsfs}
\usepackage{amscd}
\usepackage{txfonts}
\usepackage[all]{xy}

\usepackage{geometry} 
\newcommand{\piecewise}[4]{\left\{ \begin{array}{c}#1,\ {\rm if}\ #2\\ #3,\ {\rm if}\ #4\end{array}\right.}

\renewcommand{\mod}[1]{{\ifmmode\text{\rm\ (mod~$#1$)}\else\discretionary{}{}{\hbox{ }}\rm(mod~$#1$)\fi}}

\newcommand{\Ctwo}[1]{D_{#1,m}}

\begin{document}

\newcounter{claim}
\setcounter{claim}{1}

\theoremstyle{plain} \newtheorem{definition}{Definition}
\theoremstyle{plain} \newtheorem{proposition}{Proposition}
\theoremstyle{remark} \newtheorem{remark}[proposition]{Remark}
\theoremstyle{plain} \newtheorem{lemma}[proposition]{Lemma}
\theoremstyle{plain} \newtheorem{theorem}[proposition]{Theorem}
\theoremstyle{plain} \newtheorem{conjecture}[proposition]{Conjecture}
\theoremstyle{plain} \newtheorem{corollary}[proposition]{Corollary}
\theoremstyle{plain} \newtheorem{claim}[proposition]{Claim}

\begin{center}
{\LARGE{Sums and Products of Distinct Sets and Distinct \\ \vspace{0.1cm} Elements in Fields of Characteristic 0}}
\end{center}

\begin{center}
{\Large {Karsten O. Chipeniuk}} \footnote{Correspondence may be directed to: \\ Email: karstenc@math.ubc.ca \\Phone: 1 604 822 3887  \\ Fax: 1 604 822 6074 \\ The author was supported by an NSERC Postgraduate Scholarship. }
\end{center}

\begin{center}
\emph{Department of Mathematics, University of British Columbia, \\1984 Mathematics Road, Vancouver, B.C. V6T 1Z2 Canada}
\end{center}

\begin{abstract} Let $A$ and $B$ be finite subsets of an algebraically closed field $K$ of characteristic 0 such that $|B|=C|A|$.  We show the following variant of the sum product phenomena: If $|AB|<\alpha|A|$ and $\alpha,C, \alpha/C\ll \log |A|$, then $|kA+lB|\gg |A|^k|B|^l$.  This is an application of a result of Evertse, Schlickewei, and Schmidt on linear equations with variables taking values in multiplicative groups of finite rank, in combination with an earlier theorem of Ruzsa about sumsets in $\mathbb{R}^d$.  As an application of the case $A=B$ we give a lower bound on $|A^+|+|A^\times|$, where $A^+$ is the set of sums of distinct elements of $A$ and $A^\times$ is the set of products of distinct elements of $A$.
\end{abstract}


\section{Introduction}
Let $A$ be a finite subset of a commutative ring $R$.  Then we can form the sumset $2A=A+A=\{a+a': a,a'\in A\}$ and the productset $A^2=A\cdot A=\{aa':a,a'\in A\}$, as well as the iterated variant $kA$ of the sumset, for $k\in\mathbb{Z}^+$.  In addition, letting $A=\{a_1,\dots, a_n\}$, we can take the set of all sums of, and the set of all products of, distinct elements of $A$, respectively

\begin{equation}
A^+ =\left\{\sum_{i=1}^n\epsilon_ia_i: \epsilon_i\in\{0,1\},\ \forall\ i=1,\dots,n\right\}
\end{equation}
and
\begin{equation}
A^\times=\left\{\prod_{i=1}^n a_i^{\epsilon_i}: \epsilon_i\in\{0,1\},\ \forall\ i=1,\dots,n\right\}
\end{equation}
We are then led to consider 
\begin{equation}
g_R(N) = \min_{A\subset R,\ |A|=N} \{|A^+|+|A^\times|\}.
\end{equation}
Such expressions were investigated by Erd{\" o}s and Szemer{\' e}di in \cite{ES1983} in the integer setting.  They showed that
\begin{equation}
g_\mathbb{Z} (N) < N^{c\frac{\log N}{\log\log N}}
\end{equation}
for some absolute constant $c>0$.  Later Chang proved (\cite{CHA2003}) their conjecture that this essentially provides the lower bound as well; more precisely, she showed that
\begin{equation}
g_\mathbb{Z} (N) > N^{(1/8 -\epsilon)\frac{\log N}{\log\log N}}.
\end{equation}

More recently, in \cite{CHA2004}, Chang addressed a question of Ruzsa, proving that $g_\mathbb{C}(N)$ grows faster than any power of $N$,
\begin{equation}
\lim_{N\to \infty} \frac{\log(g_\mathbb{C}(N))}{\log N} = \infty.
\end{equation}

In this article we will obtain an explicit lower bound for $g_K(N)$ in an algebraically closed field $K$ of characteristic 0.  In particular, since $g_{R_2}(N)\leq g_{R_1} (N)$ whenever $R_1$ is a subring of $R_2$, this bound will hold for any field of characteristic 0.  Our result in this direction is the following.

\begin{theorem}[Lower bound on $g_K(N)$]\label{lb} For any $\epsilon >0$ and any $N$ sufficiently large we have
$$
g_K(N)\geq N^{(1/264-\epsilon)\frac{\log\log(N)}{\log\log\log(N)}}.
$$
\end{theorem}
The proof largely follows that of \cite{CHA2003}.  This approach uses a manifestation of the sum-product phenomena, namely that a small productset requires a large iterated sumset.  We will use a version which holds in the field $K$.

\begin{theorem}\label{sumprod} Let $A\subset K$ be finite, and suppose that $|A^2|\leq \alpha |A|$.  Then for any integer $h\geq 2$ we have
$$
|hA|\geq e^{(-h^{65h}(\alpha+1))} |A|^h.
$$
\end{theorem}  
This result was in essence also proved by Chang in \cite{CHA2006}, without the explicit dependence on $h$, which will be essential in the proof of \ref{lb}.  Using a result of Ruzsa (\cite{RUZ1994}, Theorem \ref{severaldimensions} below), we have also extended this result to distinct sets $A$ and $B$.  Namely, for $A,B\subset K$, define $AB$, $A+B$, and $kA+lB$ in the obvious manner.  Then we have

\begin{theorem}[Small productset implies large iterated sumset for distinct sets]\label{sumproddiff} Let $A,B \subset K$ with $|B|=C|A|$, and suppose that $|AB| < \alpha |A|$.  If 
$$
\max\left(\alpha,C,\frac{\alpha}{C}\right)\leq \frac{1}{(k+l)^{65(k+l)}} \log |A|
$$
then
\begin{equation}
|kA+lB| \gg_{k,l} |A|^k|B|^l.
\end{equation}
\end{theorem}

The proofs of Theorems \ref{sumprod} and \ref{sumproddiff} rely on bounding the number of additive tuples in $A^k$ (respectively, $A^k\times B^l$); this is approached via an induction using a result of Evertse, Schlickewei, and Schmidt \cite{ESS2002} which we next describe.

Let $K^\ast = K\setminus\{0\}$ be the multiplicative subgroup of nonzero elements in $K$.  Let $\Gamma$ be a subgroup of $(K^\ast)^d$ with rank $r$ (so the minimum number of elements from which we can generate $\Gamma$ is $r$).  For coefficients $a_1,\dots,a_d\in K$ let $A(d,r)$ denote the number of solutions $(x_1,\dots,x_d)\in\Gamma$ to

$$
a_1x_1+a_2x_2+\dots+ a_dx_d = 1
$$
which are nondegenerate (namely, no proper subsum of the left side vanishes).  Note that in the following, the bound is finite and depends only on $r$ and $d$, and not on the particular group $\Gamma$ nor the particular coefficients of the objective equation. 

\begin{theorem}[Linear equations have few solutions in a multiplicative group]\label{ss}{\rm \cite{ESS2002}}
$$
A(d,r)\leq \exp \left((6d)^{3d}(r+1)\right).
$$
\end{theorem}
We will also use two other standard tools of additive combinatorics (see, for example, \cite{TV2006}).  The first is Freiman's theorem in torsion-free groups.

\begin{theorem}[Freiman's Theorem]\label{Frei} Let $Z$ be a torsion-free abelian group, and let $A\subset Z$ with $|A+A|\leq K|A|$ for some $K\in\mathbb{R}$.  Then there exists a proper generalized arithmetic progression $P$ of dimension at most $K-1$ and size $|P|\leq C(K)|A|$ which contains $A$.  Here $C(K)$ depends only on $K$.
\end{theorem}
The second is the Pl\"unnecke-Ruzsa Inequality, in a version due to Ruzsa \cite{RUZ1996}.

\begin{theorem}[Pl\"unnecke-Ruzsa Inequality]\label{Plun} {\rm \cite{RUZ1996}} Let $Z$ be an abelian group and let $A,B\subset Z$ with $|A+B|\leq K|A|$.  Then for every $n,m\in\mathbb{N}$ we have $|nB-mB|\leq K^{n+m}|A|$.
\end{theorem}

\bigskip

\noindent
{\bf Acknowledgements:} The author would like to thank J\'ozsef Solymosi for pointing out and providing the reference \cite{ESS2002}, and Izabella {\L}aba for her constant support throughout the production of this work.

\section{Proof of Theorem \ref{sumprod}}
Let $A\subset K$ be finite with $|A|=n$ and $|A^2|\leq \alpha |A|$.

We begin by extending a definition from \cite{CHA2003}.  In particular, consider $A^\ast:=A\setminus\{0\}\subset K^\ast$.  Then we define the multiplicative dimension of $A$, denoted $\dim_\times(A)$ to be the minimal number $m$ such that $A^\ast$ is contained in a subgroup of $K^\ast$ of rank $m$.  In other words, 

\begin{equation}\label{dimtimes}
\dim_\times(A) = \min\{M: \exists\ z_1,\dots,z_M\in K^\ast\ {\rm such\ that\ } A^\ast\subset\langle z_1,\dots,z_M\rangle\}.
\end{equation}

We have the following two properties.

\begin{lemma}[Multiplicative Dimension]\label{multdim} Let $A\subset K$ be finite, and suppose $A$ has multiplicative dimension $m$.
\begin{itemize}
\item[(a)] $m\leq |(A^\ast)^2|/|A^\ast|$
\item[(b)] If $A^\ast\subset\langle z_1,\dots, z_m\rangle:=G$, there is an isomorphism $\nu: G\to(\mathbb{Z}^m,+)$.  Moreover, $\nu(A^\ast)$ is $m$-dimensional (that is, when viewed as a subset of $\mathbb{R}^m$ the smallest dimension subspace it may be contained in has full dimension)
\end{itemize}
\end{lemma}
\begin{proof}
(a) Let $\beta=|(A^\ast)^2|/|A^\ast|$.  Since $K^\ast$ is abelian, Chang's refinement of Freiman's theorem applies (\cite{CHA2002}, stated for torsion free groups in \cite{TV2006}), and we may contain $A^\ast$ in a progression in $K^\ast$ of dimension $\lfloor \beta-1\rfloor$, say
$$
A\subset \left\{z_1\prod_{i=2}^{\lfloor\alpha-1\rfloor+1} z_i^{j_i}: 0\leq j_i\leq k_i-1\right\}.
$$
But the right hand side is clearly a subset of $\langle z_1,\dots, z_{\lfloor\alpha-1\rfloor +1}\rangle$, which clearly has fewer than $\alpha$ generators.
\newline\noindent
(b) By the fundamental theorem of finitely generated abelian groups applied to $\langle z_1,\dots, z_m\rangle$, there is such an isomorphism.  Minimality of $m$ implies the full dimensionality of $A$.
\end{proof}

Note that if $0\in A$, then 
\begin{eqnarray}
|(A^\ast)^2| & = & |A^2\setminus\{0\}|\nonumber\\
& = & |A^2|-1 \nonumber\\
& \leq & \alpha|A|-1 \nonumber\\
& = & \alpha(|A^\ast|+1)-1 \nonumber\\
& < & \alpha|A^\ast|.\nonumber  
\end{eqnarray}
Hence Theorem \ref{sumprod} is an easy corollary of the following.

\begin{proposition}[Iterated Sum-Product]\label{sumprodmult} Let $A\subset K$ be finite, and suppose that $\dim_\times(A)=m$.  Then for any $h\geq 2$ we have
$$
|hA|\geq e^{(-h^{65h}(m+1))} |A|^h
$$
\end{proposition}  

To prove Proposition \ref{sumprodmult} we seek to bound the number of solutions to the equation

$$
x_1+\dots + x_{h} = x_{h+1}+\dots +x_{2h}
$$
where $x_i\in A$ for each $i$.  In the case $A=-A$, we may rewrite this equation as $x_1+\dots +x_{2h-1}+ x_{2h}=0$.  If we further assume $0\notin A$, then we are free to apply Theorem \ref{ss} to $A$.  The result is the following statement, similar to a lemma of Chang in \cite{CHA2006}:

\begin{lemma}\label{ind} Suppose that $A$ satisfies $A=-A$ and $0\notin A$.  For every $k\geq 2$ there is $n$ sufficiently large that 
\begin{itemize}
\item[(a)] The number of solutions $(y_1,\dots, y_k)\in A^k$ to $y_1+\dots+ y_k =1$ is at most $e^{k^{12k}(m+1)}n^{\lfloor k/2\rfloor}$ if $k$ is odd and $e^{k^{12k}(m+1)}n^{k/2-1}$ if $k$ is even.
\item[(b)] The number of solutions $(y_1,\dots, y_k)\in A^k$ to $y_1+\dots+y_k=0$ is at most $e^{k^{12k}(m+1)}n^{\lfloor k/2\rfloor}$.
\end{itemize}
\end{lemma}

\begin{proof}
The proof follows the same line that of the lemma in \cite{CHA2006}.  The main difference now is that we keep track of all constants.  Let $\Ctwo{t}:=\exp(t^{12t}(m + 1))$.  In addition, let 
\begin{equation}
\gamma_1(t)= \piecewise{\lfloor t/2\rfloor}{t\ {\rm is\ odd}}{t/2-1}{\rm is\ even}
\end{equation}
\begin{equation}
\gamma_0(t)= \lfloor t/2 \rfloor.
\end{equation}
Hence we are trying to show that the number of solutions to $y_1+\dots +y_k=1$ is at most $D_{k,m}n^{\gamma_1(k)}$ and the number of solutions to $y_1+\dots +y_k=0$ is at most $D_{k,m}n^{\gamma_0(k)}$.

\noindent
{\bf Base Case:} When $k=2$ we see directly that the number of solutions to $y_1+y_2=0$ is at most $n$ (Each of $|A|$ choices for $y_1$ gives one possibility for $y_2$), while the number of nondegenerate solutions to $y_1+y_2=1$ is at most $\exp((12)^6 (2m +1))$ by Theorem \ref{ss}, and the only two possible degenerate solutions are $(0,1)$ and $(1,0)$.  We have
\begin{equation}
\exp((12)^6 (2m +1))+2 \leq \exp(2^{24}(m +1)).\nonumber
\end{equation}
Hence the bound holds for $k=2$.

\noindent
{\bf Induction:} Let $k>2$ be fixed, and suppose both parts of the theorem have been proved for each integer less than $k$.  

We begin with the equation $y_1+\dots+y_k=0$, which we rewrite as $z_1+\dots+z_{k-1} = 1$ by fixing a value of $-y_k$, dividing through by it, and rearranging.  There are still $n$ possibilities for each variable in this equation, and each still falls in $G$ since $G$ is closed under division.  

First suppose k is even.  Then $k-1$ is odd, so by the inductive hypothesis there are at most $D_{k-1,m}n^{(k-2)/2}$ solutions to the latter equation, and therefore $D_{k-1,m}n^{k/2}$ solutions to the original.  Since $k$ is even, this is the desired result.

Next, if $k$ is odd, $k-1$ is even, and the latter equation has fewer than $D_{k-1,m}n^{(k-1)/2-1}$ solutions, whereby the original has fewer than $D_{k-1,m}n^{(k-1)/2}$.  This again gives the desired result.

Hence the result for zero sums holds for $k$.

To count solutions to $y_1+\dots +y_k =1$, we begin by applying the Theorem \ref{ss}.  This tells us that the number of nondegenerate solutions in the entirety of the rank $km$ group $G^k$ is bounded by $\exp((6k)^{3k}(km + 1))$.  We use the inductive hypothesis to count degenerate solutions.  But this reduces to computing, for each $t$, $2\leq t\leq k-1$, the number of solutions to the pair of equations

$$
\sum_{j=1}^{t} y_{i_j} = 0
$$
$$
\sum_{j=t+1}^{k} y_{i_j} = 1
$$
where $i_1,\dots, i_k$ is some permutation of $1,\dots, k$ with $i_1<\dots<i_t$ and $i_{t+1}<\dots<i_k$.  

Since there are $\binom{k}{t}$ choices for $\{i_1,\dots, i_t\}$, the total number of solutions is bounded via the induction by

$$
2\sum_{t=2}^{k-1}\binom{k}{t}\Ctwo{t}\Ctwo{k-t}n^{\gamma_0(t)+\gamma_1(k-t)}
$$
where we have used the extra factor of two to simply account for the small number nondegenerate solutions.  We begin by computing the exponent.  We can easily compute that for $k$ even we have
\begin{equation}
\gamma_0(t)+\gamma_1(k-t)=k/2-1.
\end{equation}
Similarly, for $k$ odd, we get
\begin{equation}
\gamma_0(t)+\gamma_1(k-t)=\piecewise{(k-1)/2}{t\ {\rm is\ even}}{(k-3)/2}{t\ \rm is\ odd}.
\end{equation}
In both cases we see that we can bound the number of solutions by
\begin{equation}
\left(2\sum_{t=2}^{k-1}\binom{k}{t}\Ctwo{t}\Ctwo{k-t}\right)n^{\gamma_1(k)}
\end{equation}
and we need only compute the constant.

Now, 
\begin{equation}
\Ctwo{t}\Ctwo{k-t} = \exp\left[(t^{12t}+(k-t)^{12(k-t)})(m + 1) \right].
\end{equation}
The exponent is maximized over all possible values of $t$ for $t=k-1$.

The entire sum is therefore bounded above by

\begin{eqnarray}
2\exp((1+(k-1)^{12(k-1)})(m + 1))\sum_{t=2}^{k-1}\binom{k}{t} \leq \exp(k^{12k}(m+1))\nonumber
\end{eqnarray}
using the fact that $\sum_{t=2}^{k-1}\binom{k}{t}<2^k$.

Hence the result follows by induction.
\end{proof}

Now, if $A\neq -A$, we simply extend $A$ to $A'=A\cup (-A)$.  This increases the size of our objective set to at most $2n$, while increasing the rank of the ambient subgroup $G$ by at most 1 (adding -1 as a generator).  Hence the rank of $G^k$ increases by at most $k$.  

There are fewer solutions to $x_1+\dots + x_{h} = x_{h+1}+\dots +x_{2h}$ in $A$ than there are in $A'$, the latter quantity being bounded by $\exp((2h)^{32h}(m+2))(2n)^h$.  

If, in addition, $0\in A$, the number of solutions we gain is certainly less than

$$
\sum_{t=1}^{2h-1}\binom{2h}{t}D_{2h-t,m+1}(2n)^{\lfloor (2h-t)/2\rfloor}
$$
Here we have set one or more variables (of which there are $2h$) equal to 0 and used Lemma \ref{ind} to count the number of solutions to the resulting equation.  We can bound this by
\begin{eqnarray}
D_{2h-1,m+1}(2n)^{h-1}\cdot 2^{2h} & \leq & 2^{3h-1}\exp((2h-1)^{16(2h-1)}(m+2))n^{h-1}\nonumber\\
& \leq & \exp((2h-1)^{16(2h-1)}(m+1)+ 5h)n^{h-1}.\nonumber
\end{eqnarray}

The total possible number of solutions after symmetrizing and adding 0, then, is
\vspace{-0.1cm}
\begin{eqnarray}
& \leq & 2^h\exp((2h)^{32h}(m+2))n^h + \exp((2h-1)^{16(2h-1)}(m+1)+5h)n^{h-1} \nonumber\\
& \leq & \exp(h^{65h}(m+1))n^h.\nonumber
\end{eqnarray}

We have therefore proved:

\begin{lemma}[Additive $2h$-tuples] The number of additive $2h$-tuples in $A$, that is the number of solutions in $A^{2h}$ to the equation $x_1+ \dots+x_h=x_{h+1}+\dots+x_{2h}$, is bounded above by $\exp(h^{65h}(m+1))n^h$.
\end{lemma}
In light of the following lemma from \cite{CHA2003} we have proved Proposition \ref{sumprodmult}.

\begin{lemma}[Cauchy-Schwarz for iterated sumsets]\label{tuples} Let $M$ denote the number of additive $2h$-tuples in a finite set $B$, that is the number of solutions in $B$ to the equation $x_1+ \dots+x_h=x_{h+1}+\dots+x_{2h}$.  Then
$$
|hB|\geq \frac{|B|^{2h}}{M}
$$
\end{lemma}

\section{Proof of Theorem 1}
We are now in a position to follow the proof in \cite{CHA2003}, using our revised definition of multiplicative dimension (\ref{dimtimes}) and the bound in Proposition \ref{sumprodmult}.  We begin by showing that a large proportion of the iterated sumset $hA$ is covered by sums of $h$ distinct elements.

\begin{lemma}[Stirling's formula applied to Lemma \ref{tuples}]\label{simpsum} Let $A\subset K$ be finite with multiplicative dimension $m$ and with $|A|$ sufficiently large.  Then for any sufficiently large $h\in\mathbb{N}$ with $h\leq |A|$ we have 
$$
|hA\cap A^+| \geq \frac{|A|^h}{\exp(h^{66h}(m+1))}.
$$
\end{lemma}
\begin{proof}
This follows exactly as in \cite{CHA2003}.  First we note that the left hand side is the number of simple sums with exactly $h$ summands.  Letting
$$
r_{hA}(x)=|\{(a_1,\dots,a_h)\in A^h: a_1+\dots+ a_h = x\}|
$$
we clearly have
$$
\binom{|A|}{h} \leq \sum_{x\in hA\cap A^+} r_{hA}(x).
$$
Using Stirling's formula in the form $N!\sim N^{N+1/2}e^{-N}$ to write
$$
\binom{|A|}{h}=\frac{|A|!}{h!(|A|-h)!}\geq C\frac{|A|^{|A|}}{h^{h+1/2}|A|^{|A|-h}}=C(|A|/h^{1+1/2h})^h.
$$
for an absolute constant $C$.  Combining the previous two relations and applying Cauch-Schwartz followed by Proposition \ref{sumprodmult} we have
\begin{eqnarray}
C(|A|/h^{1+1/2h})^h & \leq & |hA\cap A^+|^{1/2} \left(\sum_{x\in hA\cap A^+} (r_{hA}(x))^2\right)^{1/2}\nonumber\\
& \leq & |hA\cap A^+|^{1/2}(\exp(h^{65h}(m+1))|A|^h)^{1/2}.\nonumber
\end{eqnarray}
It follows that 
$$
|hA\cap A^+| \geq C\frac{|A|^h}{h^{2h+1}\exp(h^{65h}(m+1))},
$$
and absorbing the constant $C$ and the factor $h^{2h+1}$ into the exponential (for $h$ large) we have
$$
|hA\cap A^+| \geq \frac{|A|^h}{\exp(h^{66h}(m+1))}.
$$
\end{proof}

The following replaces Proposition 14 in \cite{CHA2003} for our case.

\begin{lemma}[Small mult. dim. implies large simple sum]\label{mdislarge} Let $B\subset K$, $|B|\geq \sqrt{n}$, and denote the multiplicative dimension of $B$ by $m$.  Then for any $\epsilon_1$, $0<\epsilon_1 < 1/132$, if

$$
m+1 \leq \left(\frac{1}{132} - \epsilon_1\right)\frac{\log\log(n)}{\log\log\log(n)}
$$
then
$$
g(B)\geq n^{\epsilon_1\frac{\log\log(n)}{\log\log\log(n)}}.
$$
\end{lemma}

\begin{proof}
We clearly have
$$
g(B)> |B^+| \geq |hB \cap B^+|
$$
for any $h\in\mathbb{N}$.  Applying Lemma \ref{simpsum} we have 
$$
g(B)> \frac{|B|^h}{\exp(h^{66h}(m+1))}.
$$
Now, we take $h=\lfloor \frac{1}{66}\frac{\log\log(n^{1/2})}{\log\log\log(n)}\rfloor$.  Then
$$
66h\log(h) \leq \log\log(n^{1/2})
$$
so
$$
\exp(h^{66h}(m+1)) \leq n^{(m+1)/2}.
$$
But then we have
$$
g(B) > \frac{|B|^h}{n^{(m+1)/2}} \geq n^{h/2-(m+1)/2}
$$
Recalling our condition on $m$ and our choice of $h$ we have the result.
\end{proof}

This last lemma reduces us to considering the case where all large subsets of $A$ have large multiplicative dimension,

$$
m\geq \lfloor \left(\frac{1}{132} - \epsilon_1\right)\frac{\log\log(n)}{\log\log\log(n)}\rfloor.
$$

Now, recalling Lemma \ref{multdim}, giving $\nu(A)^+$ the obvious meaning we note that 
\begin{equation}
|A^\times|=|\nu(A)^+|.
\end{equation}
To complete the proof of the theorem we follow very closely the argument given in \cite{CHA2003}.

\bigskip

We divide $A$ into sets $B_1,\dots,B_{\lfloor \sqrt{n}\rfloor}$, each with size $|B_i|\geq \sqrt{n}$, and then denote $A_s=\cup_{i=1}^s B_i$.  We let $\epsilon_2$ satisfy $0<\epsilon_2<1/2$, and we let

$$
\rho = 1+ n^{-1/2+\epsilon_2}.
$$

\bigskip

The proof now splits into a trivial case, followed by a complementary case in which we are able to effectively use the large multiplicative dimension of the $B_i$s.

First, if $|\nu(A_s\cup B_{s+1})^+|> \rho|\nu(A_s)^+|$ for every $s$, we can begin with $A_1=B_1$ and iterate, gaining a factor of $\rho$ each time, to get

$$
|\nu(A)^+|>\rho^{\lfloor \sqrt{n}\rfloor -1}|\nu(B_1)| > \rho^{\sqrt{n}-2}\sqrt{n}.
$$
Hence, using (for $x$ small) $\log(1+x)>x - x^2/2 > x/2$ we have
\begin{eqnarray}
g(A) & > & |A^\times|\ =\ |\nu(A)^+| \nonumber\\
& > & \exp((\sqrt{n}-2)\log(\rho) + \frac{1}{2}\log(n)) \nonumber\\
& > & \exp((\sqrt{n}-2)(1/2)n^{-1/2+\epsilon_2} + \frac{1}{2}\log(n))\nonumber\\
& > & \exp((1/2)n^{\epsilon_2})\nonumber
\end{eqnarray}
The last bound is clearly much better than what we are trying to prove.

\bigskip

We are therefore reduced to the case where $|\nu(A_s\cup B_{s+1})^+|\leq \rho|A_s^+|$ for some s.  Let $m=\dim_\times(B_{s+1})$, so $m\geq \lfloor \left(\frac{1}{132} - \epsilon_1\right)\frac{\log\log(n)}{\log\log\log(n)}\rfloor$.

Now, since the sets $B_i$ are disjoint, the sets $\nu(B_i)$ are as well, so that
\begin{equation}
\nu(A_s\cup B_{s+1})^+ = (\nu(A_s)\cup\nu(B_{s+1}))^+ = \nu(A_s)^+ + \nu(B_{s+1})^+.
\end{equation}  
We therefore have 
\begin{equation}
|\nu(A_s)^+ + \nu(B_{s+1})^+|\leq \rho |\nu(A_s)^+|,
\end{equation}
so by Pl{\" u}nnecke's Inequality (Theorem \ref{Plun}) for any $h\in\mathbb{N}$ we have

$$
|(h+1)\nu(B_{s+1})^+ - \nu(B_{s+1})^+| \leq \rho^{h+2}|\nu(A_s)^+|.
$$ 
But the left hand side is of course larger than 
\begin{eqnarray}
|h\cdot\nu(B_{s+1})^+| & \geq & |\nu(B_{s+1})^+[h]| \nonumber\\
& \geq & h^m. \nonumber
\end{eqnarray}
where setting $\nu(B_{s+1})=\{b_1,\dots,b_k\}$ we have defined
$$
\nu(B_{s+1})^+[h]:=\left\{\sum_{i=1}^k\epsilon_ib_i: \epsilon_i\in\{0,1,\dots, h\},\ i=1,\dots,k\right\}
$$
and in which the second inequality comes from the simple sum of a basis of $\mathbb{Z}^m$ chosen from $\nu(B_{s+1})$ (via Lemma \ref{multdim}).  We take $h=\lfloor n^{1/2-\epsilon_2}\rfloor$, so that

\begin{eqnarray}
\rho^{h+2} & \leq & (1+n^{-1/2+\epsilon_2})^{n^{1/2-\epsilon_2}+2}\nonumber\\
& < & \exp(n^{-1/2+\epsilon_2}(n^{1/2-\epsilon_2}+2))\nonumber\\
& < & e^3.\nonumber
\end{eqnarray}
Combining, we have

\begin{eqnarray}
g(A) & > & g(A_s) \nonumber\\
& > & |\nu(A_s)^+|\nonumber\\
& > & e^{-3}h^m\nonumber\\
& > & e^{-3}\lfloor n^{1/2-\epsilon_2}\rfloor^{\lfloor \left(\frac{1}{132} - \epsilon_1\right)\frac{\log\log(n)}{\log\log\log(n)}\rfloor}.\nonumber
\end{eqnarray}
This proves the proposition.

\section{Sums of Distinct Sets With Small Productset}

We now prove Theorem \ref{sumproddiff}.  Let $A,B\subset K$ be finite, with $|B|=C|A|$, and suppose that $|AB|< \alpha |A|$.  Fix intergers $k$ and $l$, and assume that 
$$
\max\left(\alpha,C,\frac{\alpha}{C}\right)\leq \frac{1}{(k+l)^{65(k+l)}} \log |A|.
$$

For $S\subset K$ finite, we will denote by $G_S$ some fixed multiplicative group in $K^\ast$ of rank $\dim_\times S$ which contains $S^\ast$ (see Lemma \ref{multdim}).

The strategy is to use the condition $|AB|<\alpha|A|$ to bound the multiplicative dimensions of $|A|$ and $|B|$ in terms of $\alpha$ and $C$.  Our main tool will be the following result of Ruzsa \cite{RUZ1994}.

\begin{theorem}[Sumsets in $\mathbb{R}^n$]\label{severaldimensions}{\rm \cite{RUZ1994}} Let $n\in\mathbb{Z}^+$, and let $X,Y\subset \mathbb{R}^n$, $|X|\leq |Y|$, and suppose $\dim(X+Y)=n$.  Then we have
\begin{equation}
|X+Y|\geq |Y| + n|X| - \frac{n(n+1)}{2}.
\end{equation}
\end{theorem}

Now, suppose $0\notin A\cup B$, and let $D= \dim_\times (A\cup B)$.  Then by Lemma \ref{multdim} there is an isomorphism $\nu: G_{A\cup B} \to (\mathbb{Z}^D,+)$.  Then we can compute $\nu(A)+\nu(B)$, and set $d = \dim(\nu(A)+\nu(B)) \geq \dim_\times (AB)$.  Now, we may have $d<D$, so we cannot immediately apply Theorem \ref{severaldimensions}.  However, $\nu(A)+\nu(B)$ contains translates $\nu(A)+\beta$ and $\alpha+\nu(B)$ for $\alpha\in \nu(A),\ \beta\in \nu(B)$.  Letting $\mathbb{A}^d$ be the real affine space containing $\nu(A) + \nu(B)$, we can provide an isomorphism $\eta:\mathbb{A}^d\to \mathbb{R}^d$.  Setting $X= \eta(\nu(A)+\beta)$ and $Y=\eta(\alpha+\nu(B))$ we find that

\begin{eqnarray}
|X+Y| & = & |\eta(\nu(A)+\nu(B)+\alpha+\beta)| \nonumber\\
& = & |\nu(A)+\nu(B)+\alpha+\beta|\nonumber\\
& = & |\nu(A)+\nu(B)| \nonumber\\
& = & |AB|
\end{eqnarray}
and $\dim(X+Y)=d$.  We therefore have

\begin{corollary}\label{severalmultdim} Let $A,B$, and $d$ be as above.  Then we have
\begin{equation}
|AB|\geq |B| + d|A| - \frac{d(d+1)}{2}.
\end{equation}
\end{corollary}
Now, since $\dim(X+Y)\leq \dim(X)+\dim(Y)$, we have $|A|+|B|=|X|+|Y|\geq d+2$.  
We can now derive

\begin{lemma}\label{multdimbound} Let $A,B$, and $d$ be as above, and let $K\geq 1$ be such that $|A|+|B|\geq K(d+2)$. Let $m=\max(\dim_\times (A), \dim_\times (B))$.  Then
\begin{itemize}
\item[(a)] If $K>C\geq 1$ we have
\begin{equation}
m<\frac{\alpha - C}{1-\frac{C}{K}}
\end{equation}
\item[(b)] If $C<1$, we have
\begin{equation}
m<\frac{(\alpha - 1)}{C\left(1-\frac{1}{K}\right)}
\end{equation}
\end{itemize}
\end{lemma}

\begin{proof}
Since $AB$ contains (multiplicative) translates of both $A$ and $B$, we see that $d\geq \max(\dim_\times (A), \dim_\times (B))$.  
\begin{itemize}
\item[(a):] If $C\geq 1$, then by Corollary \ref{severalmultdim} we have
$$
\alpha|A| > |AB| \geq |B| + d|A| - \frac{d(d+1)}{2}.
$$
Rearranging,
$$
\alpha \geq C+d-\frac{d(d+1)}{2|A|}.
$$
Now, we have $2C|A|= 2|B| > |A|+|B|$, so this gives
$$
\alpha - C \geq d-\frac{Cd(d+1)}{K(d+2)}
$$
and we see the result.
\item[(b):] If $C<1$, then
$$
\alpha|A| > |A| + d|B| - \frac{d(d+1)}{2}.
$$
Hence
$$
\frac{1}{C}(\alpha - 1) > d - \frac{d(d+1)}{2C|A|}
$$
and substituting for $2C|A|$ in the last term we have the result.
\end{itemize}
\end{proof}

The singular behaviour when $K=C$ or $K=1$ will not be a problem for the application, as the following lemma shows.

\begin{lemma}[Multiplicative bases have large productset]\label{bases} Let $d$ be a large integer.  Suppose that $X,Y\subset \mathbb{R}^d$ satisfy $|Y|=C|X|$, $C\leq \log |X|$, and $\dim(X+Y)=d$.  Set $|X|+|Y|= K(d+2)$.  Then if $K\leq \log |X|$ we have 
\begin{equation}
|X+Y|\geq |X||Y|/(2^{10}\log^2 |X|).
\end{equation}
\end{lemma}
\begin{proof}  
Since $\dim(X\cup Y)\geq \dim(X+Y)$, there are linearly independent vectors $x_1,\dots, x_r$ and $y_1,\dots, y_{d-r}$ for some $1\leq r\leq d$ and points $x_0$, $y_0$ such that $x_0+x_i\in X$ for $1\leq i \leq r$ and $y_0+y_j\in Y$ for $1\leq j\leq d-r$.  

First, suppose that $r\geq d/2$.  Then let $Y'$ be any subset of $Y$ with $|Y'|= \lfloor |Y|/(16\log |X|) \rfloor$.  Note that $|Y|\geq |X|$, so $|Y'|\neq 0$, but

$$
|Y'| \leq (d+2)/8 = d/8 + 1/4 \leq d/4.
$$
Hence we may choose $X'\subset \{ x_1,\dots, x_r\}$ such that $|X'| = \lfloor d/4 \rfloor \geq |X|/(16\log |X|)$ and such that the spans of $X'$ and $Y'$ do not intersect.  Since $x+y=x'+y'$ forces $y-y'=x'-x$, it follows that

$$
|X+Y|\geq |X'+Y'| \geq (|X|/(2^5\log |X|))(|Y|/(2^5\log |X|)).
$$

If $r\leq d/2$, then $d-r\geq d/2$, and the argument is the same with the sets exchanged.
\end{proof}

Since we have assumed that $\alpha, C < \log |A|$, we can use Lemma \ref{bases} with the assumption on $|AB|$ to bound $K$ away from $C$ when $C\geq 1$, and away from $1$ when $C<1$.  To proceed we begin by noting the following analogy of \ref{tuples}.

\begin{lemma}[Cauchy-Schwartz for sumsets of distinct sets]\label{CSdistinct} Let $M$ denote the number of additive $2(k+l)$-tuples in $A^{2k}\times B^{2l}$.  Then
\begin{equation}
|kA+lB|\geq \frac{|A|^{2k}|B|^{2l}}{M}
\end{equation}
\end{lemma}

\begin{proof}
Let $r_{kA+lB}(x)= |\{(a_1+\dots+a_k+b_1+\dots+b_l)\in A^k\times B^l : a_1+\dots +a_k+b_1+\dots +b_l = x\}|$.  Then by the Cauchy-Schwartz inequality we have
\begin{eqnarray}
|A|^{2k}|B|^{2l} = \sum_{x\in kA+lB} r_{kA+lB}(x) & \leq & |kA+lB|^{1/2} \left(\sum_{x\in kA+lB} (r_{kA+lB}(x))^2\right)^{1/2} \nonumber\\
& = & |kA+lB|^{1/2} M^{1/2}\nonumber\\
\end{eqnarray}
\end{proof}
Using Chang's version of the induction lemma with $A\cup B$, we can immediately obtain a version of Theorem \ref{sumproddiff} provided that $C$ is absolutely bounded.  If we wish to obtain the finer control over $C$, however, we must refine the argument by proceeding through the proof of the lemma while remaining attentive to which variables lie in $A$ and which lie in $B$. 

We define $\gamma_0$, $\gamma_1$, and $D_{t,m}$ as we did in the proof of Lemma \ref{ind}.  We summarize the following additivity properties for use in the sequel.

\begin{lemma}[Additivity properties of $\gamma_0$ and $\gamma_1$]\label{gammas} For positive integers $k,l$ such that $k+l>2$ we have
\begin{itemize}
\item[(a)] $\gamma_1(k+l-1)= \gamma_0(k+l)-1$
\item[(b)] If $k,l\geq 2$ and $k,l$ are not both odd, then $\gamma_1(k+l-1) = \gamma_0(k)+\gamma_0(l) - 1$.
\item[(c)] $\gamma_0(k)+\gamma_1(l)\leq \gamma_1(k+l)$
\end{itemize}
\end{lemma}

\begin{proof}
We may compute all three directly by separating into cases based on the parity of $k$ and $l$.
\end{proof}
Let
$$
\mu_0(k+l)=\piecewise{C^{\lfloor l/2\rfloor}D_{k+l,m}|A|^{\gamma_0(k+l)}}{k+l\ {\rm is\ odd}}{C^{\lfloor l/2\rfloor}|A|^{\gamma_0(k+l)} + C^{\lfloor l/2\rfloor}D_{k+l,m}|A|^{\gamma_0(k+l)-1}}{k+l\ {\rm is\ even}}
$$
$$
\mu_1(k+l)=\piecewise{C^{\lfloor l/2\rfloor}|A|^{\gamma_1(k+l)} + C^{\lfloor l/2\rfloor}D_{k+l,m}|A|^{\gamma_1(k+l)-1}}{k+l\ {\rm is\ odd}}{C^{\lfloor l/2\rfloor}D_{k+l,m}|A|^{\gamma_1(k+l)}}{k+l\ {\rm is\ even}}.
$$
We can now prove:

\begin{lemma}[Small multiplicative dimension implies few solutions]\label{inddistinct} Suppose $A,B\subset K$ satisfy $A=-A, B=-B$, and $0\notin A\cup B$.  Let $m=\max(\dim_\times(A),\dim_\times(B))$.  Set $|B|=C|A|$.  For $k,l\in\mathbb{N}$, let $\sigma_0(k,l)$ denote the number of solutions to $a_1+\dots+a_k+b_1+\dots+b_l=0$ and let $\sigma_1(k,l)$ denote the number of solutions to $a_1+\dots+a_k+b_1+\dots+b_l=1$.
\begin{itemize}
\item[(a)] $\sigma_0(k,l)\leq \mu_0 (k+l)$
\item[(b)] In addition, if $k,l\geq 2$ are not both odd, then for $k+l$ odd we have
$$
\sigma_0(k,l)\leq D_{k+l,m}C^{\lfloor l/2\rfloor}|A|^{\gamma_0(k)+\gamma_0(l)-1}
$$ 
and for $k+l$ even we have
$$
\sigma_0(k,l)\ll_{k+l} C^{\lfloor l/2\rfloor}|A|^{\gamma_0(k)+\gamma_0(l)} + C^{\lfloor l/2 \rfloor}D_{k+l,m}|A|^{\gamma_0(k)+\gamma_0(l)-1}.
$$
\item[(c)] $\sigma_1(k,l)\leq \mu_1(k+l)$
\end{itemize}
\end{lemma}
\begin{proof}
We may assume that $k,l>0$ since the remaining cases are covered by the version of Lemma \ref{ind} in \cite{CHA2006}.
\smallskip

\noindent
{\bf Base Case:} When $k=1,l=1$, we see that the number of solutions to $a_1+b_1=0$ is at most $|A|$ (Each of $|A|$ choices for $a_1$ gives one possibility for $b_1$).  Meanwhile, the number of nondegenerate solutions to $a_1+b_1=1$ is at most $\exp((12^6(2m+1))$ by Theorem \ref{ss}, and the only two possible degenerate solutions are $(0,1)$ and $(1,0)$.  We have
\begin{equation}
\exp((12)^6 (2m +1))+2 \leq \exp(2^{24}(m +1)).\nonumber
\end{equation}
Hence the bound holds for $k+l=2$.

\smallskip

\noindent
{\bf Induction:} Fix $k,l\in \mathbb{Z}^+$, and suppose that all parts of the lemma have been proved for pairs $k',l'$ with $k'+l'<k+l$.  
\begin{itemize}
\item[(a):] We can rewrite $a_1+\dots+a_k+b_1+\dots+b_l=0$ as $a_2'+\dots+a_k'+b_1'+\dots+b_l'=1$, where the variables $a_i',b_j'$ are constrained to sets of size $|A|$ and $|B|$ respectively (with no change in multiplicative dimension), by fixing a value of $-a_1$, dividing through by it, and rearranging.  By the inductive hypothesis, this new equation has fewer than $\mu_1(k+l-1)$ solutions, whereby our target equation has at most
$$
|A|\mu_1(k+l-1)\leq \mu_0(k+l)
$$
by Lemma \ref{gammas} (a), as desired.
\item[(b):] Let $k,l\geq 2$, $k$ and $l$ not both odd.  We proceed exactly as above, but apply Lemma \ref{gammas} (b) to $\mu_1(k+l-1)$ to obtain the result.
\item[(c):] To count solutions to $a_1+\dots+a_k+b_1+\dots+b_l=1$, we begin by applying the Theorem \ref{ss}.  This tells us that the number of nondegenerate solutions in the entirety of the group $G_A^k\times G_B^l$ is bounded by $\exp(6(k+l)^{3(k+l)}((k+l)m+1))$.  

We use the inductive hypothesis to count degenerate solutions.  This reduces to computing, for each quadruple $k_1,l_1,k_2,l_2\in\mathbb{Z}^+$ such $k_1+k_2=k$, $l_1+l_2=l$, and $k_1+l_1 \geq 2$, the number of solutions to the pair of equations
$$
\sum_{r=1}^{k_1} a_{i_r} + \sum_{r=1}^{l_1} b_{j_r} = 0
$$
$$
\sum_{r=k_1+1}^{k} a_{i_r} + \sum_{j=l_1+1}^{l} b_{j_r} = 1
$$
where $i_1,\dots, i_k$ is some permutation of $1,\dots, k$ with $i_1<\dots<i_{k_1}$ and $i_{k_1+1}<\dots<i_k$, and $j_1\dots, j_l$ is some permutation of $1,\dots, l$ with $i_1<\dots<i_{l_1}$ and $i_{l_1+1}<\dots<i_l$.

The number of solutions is therefore bounded by

$$
2\sum_{k_1,l_1} \binom{k}{k_1}\binom{l}{l_1} \mu_0(k_1+l_1)\mu_1(k_2+l_2)
$$
where again we have used the extra factor of two to account for the nondegenerate solutions.  

Now, if $k+l$ is odd, then we either have $k_1+l_1$ even and $k_2+l_2$ odd, in which case 
$$
\mu_0(k_1+l_1)\mu_1(k_2+l_2) \leq \mu_1(k+l)
$$
by an application of Lemma \ref{gammas} (c), or else $k_1+l_1$ is odd and $k_2+l_2$ is even, in which case
$$
\mu_0(k_1+l_1)\mu_1(k_2+l_2) \leq \frac{k+l-3}{2} < \mu_1(k+l).
$$ 
We note that to combine the constants $D_{k_1+l_1,m}D_{k_2+l_2,m}$ here we have set $t=k_1+l_1$, so $k_2+l_2= k+l-t$, and have used the analysis from Lemma \ref{ind}.  The result now follows for $k+l$ odd on factoring out $\mu_1(k+l)$ and computing $\sum_{k_1,l_1} \binom{k}{k_1}\binom{l}{l_1} = 2^{k_1+l_1}$.

If $k+l$ is even, we proceed similarly, however in each decomposition $k_1+l_1$ and $k_2+l_2$ will have identical parity, leading to the loss of exponent in $\mu_1(k+l)$.

The result now follows by induction.
\end{itemize}
\end{proof}
Now, for arbitrary sets $A,B\subset K$ satisfying $|AB|<\alpha |A|$, we may apply Lemma \ref{inddistinct} to $A'=(A\cup(-A))\setminus \{0\}$ and $B'=(B\cup(-B))\setminus \{0\}$.  As in Section 2, we may then bound the number of solutions gained by adding 0 back to $A$ and $B$.  The numerics are identical to those previous, and the final result is

\begin{lemma}[Additive tuples in distinct sets]\label{addtuplesbound} Let $A,B\subset K$ satisfy $|B|=C|A|$ and $|AB| < \alpha |A|$.  Then the number of additive $2k+2l$-tuples $M$ in $A^{2k}\times B^{2l}$ satisfies
$$
M\ll_{k+l} |A|^k|B|^l + e^{((k+l)^{65(k+l)}\alpha)}|A|^{k-1}|B|^l
$$
if $C\geq 1$ and
$$
M\ll_{k+l} |A|^k|B|^l + e^{((k+l)^{65(k+l)}\frac{\alpha}{C})}|A|^{k-1}|B|^l
$$
if $C<1$.
\end{lemma}
\noindent Theorem \ref{sumproddiff} now follows by Lemma \ref{CSdistinct}.

\bibliographystyle{plain}
\bibliography{biblio}

\end{document}